\documentclass[oneside,english]{amsart}
\usepackage[T1]{fontenc}
\usepackage[latin9]{inputenc}
\usepackage{amsthm}
\usepackage{amstext}
\usepackage{amssymb}
\usepackage{esint}
\makeatletter
\numberwithin{equation}{section}
\numberwithin{figure}{section}
\theoremstyle{plain}
\newtheorem{thm}{\protect\theoremname}[section]

\theoremstyle{plain}
\theoremstyle{definition}

\theoremstyle{plain}

\theoremstyle{plain}

\theoremstyle{plain}
\makeatother

\usepackage{babel}
\providecommand{\definitionname}{Definition}
\providecommand{\lemmaname}{Lemma}
\providecommand{\theoremname}{Theorem}
\providecommand{\corollaryname}{Corollary}
\providecommand{\remarkname}{Remark}
\providecommand{\propositionname}{Proposition}
	
\DeclareMathOperator{\loc}{loc}
\DeclareMathOperator{\supp}{supp}

\DeclareMathOperator{\cp}{cap}

\DeclareMathOperator{\BMO}{BMO}
\DeclareMathOperator{\Lip}{Lip}

\begin{document}



\title[Composition operators and $\BMO$-quasiconformal mappings]{Composition operators on Hardy-Sobolev spaces and $BMO$-quasiconformal mappings}

\author{Alexander Menovschikov and Alexander Ukhlov}

\begin{abstract}
In this paper we consider composition operators on Hardy-Sobolev spaces in connections with $BMO$-quasiconformal mappings. Using the duality of Hardy spaces and $\BMO$-spaces we prove that $BMO$-quasiconformal mappings generate bounded composition operators from Hardy-Sobolev spaces to Sobolev spaces.

\end{abstract}

\maketitle

\section{Introduction }

Composition operators on Sobolev spaces arise in the work by V.~Maz'ya \cite{M69} in connection with the isoperimetric problem as operators generated by sub-areal mappings. In this pioneering work it was established a connection between geometrical properties of mappings and the corresponding Sobolev spaces.
In the present paper we consider composition operators on Hardy-Sobolev spaces generated by $BMO$-quasiconformal mappings. 
Let Hardy-Sobolev spaces $H^{1,n}_r(\Omega)$ are defined in Lipschitz bounded domains in $\Omega\subset \mathbb R^n$, 
Sobolev spaces $L^{1,n}(\widetilde{\Omega})$ are defined in bounded domains in $\widetilde{\Omega}\subset \mathbb R^n$ and $\varphi:\Omega\to\widetilde{\Omega}$ is a $BMO$-quasiconformal mapping. Then the inequality
$$
\|f\circ \varphi^{-1}\mid L^{1,n}(\widetilde{\Omega})\|\leq 
\|Q \mid {\BMO_z({\Omega})}\|^{\frac{1}{n}} \|f \mid {H^{1,n}_r({\Omega})}\|,
$$
where a measurable function $Q:\Omega\to\mathbb R$ be such that a quasiconformal distortion $K(\varphi)\leq Q$ a.~e. in $\Omega$ \cite{MRSY09}, holds for any Lipschitz function $f\in \Lip ({\Omega})$.

$\BMO$-quasiconformal mappings generalize the notion of quasiconformal mappings, because $K$-quasiconformal mappings are $\BMO$-quasiconformal mappings with $Q:=K\in\BMO(\Omega)$.
Composition operators on Sobolev spaces in connections with quasiconformal mappings were considered in \cite{VG75} in the frameworks of Reshetnyak's problem (1968). Note that this problem arises to quasiconformal mappings and Royden algebras \cite{L71,N60}.
In \cite {VG75} it was proved that a homeomorphism $\varphi:\Omega\to\widetilde{\Omega}$, where
$\Omega$, $\widetilde{\Omega}$ are domains in $\mathbb R^n$, generates by the composition rule $\varphi^{\ast}(f)=f\circ\varphi$ the bounded operator on Sobolev spaces
$$
\varphi^{\ast}: L^{1,n}(\widetilde{\Omega}) \to L^{1,n}(\Omega),
$$
if and only if $\varphi$ is a quasiconformal mapping. In the case of Sobolev spaces $L^{1,p}(\widetilde{\Omega})$ and $L^{1,p}(\Omega)$, $p\ne n$, the analytic description was obtained in \cite{V88} using a notion of mappings of finite distortion introduced in \cite{VGR}: a weakly differentiable mapping is called a mapping of finite distortion if $|D\varphi(x)|=0$ a.~e. on the set $Z=\{x\in\Omega: J(x,\varphi)=0\}$. 
In \cite{GGR95} characterizations of composition operators in geometric terms for $n-1<p<\infty$ were obtained. 

The case of Sobolev spaces $L^{1,p}(\widetilde{\Omega})$ and $L^{1,p}(\Omega)$, $q<p$, is more complicated and in this case the composition operators theory is based on the countable-additive set functions, related with norms of composition operators and introduced in \cite{U93} (see also \cite{VU02}).
The main result of \cite{U93} gives analytic and capacitary characterizations of composition operators on Sobolev spaces (see, also \cite{VU02}) in terms of mappings of finite distortion \cite{HK14,VGR}. Multipliers theory has been applied to the change of variable problem in Sobolev spaces in \cite{MS86}.

In the last decade the composition operators theory has been considered on some generalizations of Sobolev spaces, such as Besov spaces and Triebel-Lizorkin spaces, \cite{HK13,KKSS14,KYZ11,KXZZ17,OP17}. These types of composition operators have applications to the Calder\'on inverse conductivity problem \cite{C80}. Composition operators on Sobolev spaces over Banach function spaces (such as Orlicz, Lorentz, variable exponents etc.) have been considered in \cite{ADV09, HK12, HKM14, K14, R17, MA16, MA17}.

Remark that composition operators on Sobolev spaces have significant applications to the Sobolev embedding theory \cite{GGu,GU} and to the spectral theory of elliptic operators, see, for example, \cite{GU16,GU17,GPU18_3}. In some cases the composition operators method allows one to obtain better estimates than the classical L.~E.~Payne and H.~F.~Weinberger estimates in convex domains \cite{PW}.

The notion of $Q$-mappings was introduced in \cite{MRSY01} (see also \cite{MRSY04}--\cite{MRSY09}).
Recall that a homeomorphism $\varphi: \Omega\to\Omega'$ of domains
$\Omega,\Omega'\subset \mathbb{R}^n$ is called a $Q$-homeomorphism with a non-negative measurable function $Q$, if
$$
M\left(\varphi \Gamma\right)\leqslant \int\limits_{\Omega} Q(x)\cdot
\rho^{\nu}(x)dx
$$
for every family $\Gamma$ of rectifiable paths in $\Omega$ and every admissible function $\rho$ for $\Gamma$.

The $Q$-mappings with a function $Q$ belongs to the $A_n$-Muckenhoupt class are inverse to homeomorphisms generating bounded composition operators on the weighted Sobolev spaces \cite{UV08} (see, also \cite{V20}). In the case $Q\in \BMO(\Omega)$ we have a class of $BMO$-quasiconformal mappings \cite{MRSY09,RSY01}. Note that $\BMO$-quasiconformal mappings have significant applications in the Beltrami equation theory \cite{GRSY12}.

The aim of the present article is to study $Q$-mappings with $Q\in\BMO$ in connection with composition operators on Sobolev-type spaces. This leads us to consider composition operators on Hardy-Sobolev spaces.

The theory of Hardy spaces on the Euclidean space $\mathbb{R}^n$, arise in the work by E.~M.~Stein and G.~Weiss in \cite{SW60}. Later, C.~Fefferman and E.~M.~Stein \cite{FS72} systematically developed the real-variable theory for Hardy spaces $H^p(\mathbb R^n)$ with $p\in(0, 1]$, which plays an important role in various fields of analysis (see, for example, \cite{St93}). Hardy spaces and $\BMO$-spaces on domains of $\mathbb{R}^n$ were considered in \cite{C94,CKS93}. The current state of the art and references to applications of Hardy spaces on domains of $\mathbb{R}^n$ the reader will find in \cite{CJY16}. Composition operators on Hardy and Hardy-Sobolev spaces of analytic functions have been intensively studied for a long time and can be found, for example in \cite{CM95, HCH15}.

\section{Hardy-Sobolev spaces}

\subsection{Sobolev spaces}

Let $E$ be a measurable subset of $\mathbb R^n$, $n\geq 2$. The Lebesgue space $L^p(E)$, $1\leq p<\infty$, is defined as a Banach space of $p$-summable functions $f:E\to \mathbb R$ equipped with the following norm:
$$
\|f\mid L^p(E)\|=
\biggr(\int\limits_E|f(x)|^p\,dx\biggr)^{\frac{1}{p}},\,\,\,1\leq p<\infty.
$$

If $\Omega$ is an open subset of $\mathbb R^n$, the Sobolev space $W^{1,p}(\Omega)$, $1\leq p<\infty$, is defined \cite{M}
as a Banach space of locally integrable weakly differentiable functions
$f:\Omega\to\mathbb{R}$ equipped with the following norm: 
\[
\|f\mid W^{1,p}(\Omega)\|=\| f\mid L^p(\Omega)\|+\|\nabla f\mid L^p(\Omega)\|,
\]
where $\nabla f$ is the weak gradient of the function $f$, i.~e. $ \nabla f = (\frac{\partial f}{\partial x_1},...,\frac{\partial f}{\partial x_n})$.

The homogeneous seminormed Sobolev space $L^{1,p}(\Omega)$, $1\leq p<\infty$, is defined as a space
of locally integrable weakly differentiable functions $f:\Omega\to\mathbb{R}$ equipped
with the following seminorm: 
\[
\|f\mid L^{1,p}(\Omega)\|=\|\nabla f\mid L^p(\Omega)\|.
\]

\subsection{Hardy and Hardy-Sobolev spaces}

Let us recall the classical definition of Hardy spaces $H^1(\mathbb{R}^n)$. 
Let $\Phi \in \mathcal{S}(\mathbb{R}^n)$ be a function such that $\int_{\mathbb{R}^n} \Phi(x) \, dx = 1$. For all $t>0$, define $\Phi_t(x)= t^{-n}\Phi(x/t)$ and the vertical maximal function
$$
\mathcal{M}f(x)= \sup\limits_{t>0} |\Phi_t * f(x)|.
$$

Let a function $f \in L^1_{\loc}(\mathbb{R}^n)$, then $f$ is said to be in $H^1(\mathbb{R}^n)$ if $\mathcal{M}f \in L^1(\mathbb{R}^n)$. The Hardy space $H^1(\mathbb{R}^n)$ is equipped with the norm
$$
\|f \mid H^1(\mathbb{R}^n) \| := \|\mathcal{M}f \mid L^1(\mathbb{R}^n)\|.
$$

There are several definitions of Hardy spaces \cite{C94,CKS93,M90} and Hardy-Sobolev spaces on domains $\Omega \subset \mathbb{R}^n$ (see, e.g. \cite{ART05, CJY16}). Following \cite{ART05} we define two type of Hardy spaces on Lipschitz domains in $\mathbb R^n$.
The Hardy space $H^1_z(\Omega)$ is defined as a space of functions 
$f \in H^1(\mathbb R^n)$, such that $\supp f \subset \overline{\Omega}$. Endowed with the norm 
$$
\|f\mid H^1_z(\Omega)\|:=\|f \mid H^1(\mathbb R^n)\|,
$$
it is a Banach space. 

The Hardy space $H^1_r(\Omega)$ is defined as a space of functions 
$f$ which are restrictions to $\Omega$ of functions $F\in H^1(\mathbb R^n)$. If $f\in H^1_r(\Omega)$ then 
$$
\|f\mid H^1_r(\Omega)\|:=\inf \|F\mid H^1(\mathbb R^n)\|,
$$
where the infimum is taken over all functions $F\in H^1(\mathbb R^n)$ such that $F\vert_{\Omega}=f$.
The space $H^1_r(\Omega)$ equipped with this norm is a Banach space.
In \cite{M90}, it was shown that $H^1_r(\Omega)$ can be define in terms of maximal function: $\|f \mid H^1_r(\Omega)\| = \| \mathcal{M}_\Omega f \mid L_1(\Omega)\|$,
$$
\mathcal{M}_\Omega f(x)= \sup\limits_{t \leq d(x, \partial \Omega)} |\Phi_t * f(x)|.
$$

We define the Hardy-Sobolev space $HS^{1,p}_r(\Omega)$ ($HS^{1,p}_z$), $1\leq p<\infty$, as a space of weakly differentiable functions $f \in L^p(\Omega)$ such that $|\nabla f|^p \in H^1_r(\Omega)$ ($|\nabla f|^p \in H^1_z(\Omega)$) and equipped with the norms
$$
\|f \mid HS^{1,p}_r(\Omega)\| := \|f \mid L^p(\Omega)\| + \||\nabla f|^p \mid H^1_r(\Omega)\|^{\frac{1}{p}},
$$ 
$$
\|f \mid HS^{1,p}_z(\Omega)\| := \|f \mid L^p(\Omega)\| + \||\nabla f|^p \mid H^1_z(\Omega)\|^{\frac{1}{p}}.
$$

The homogeneous Hardy-Sobolev space $H^{1,p}_r(\Omega)$ ($H^{1,p}_z(\Omega)$), $1\leq p<\infty$, we define as a space of locally integrable weakly differentiable functions $f:\Omega\to\mathbb{R}$ equipped with the following seminorms: 
\[
\|f\mid H^{1,p}_r\Omega)\|:=\||\nabla f|^p \mid H^1_r(\Omega)\|^{\frac{1}{p}} \,\, \text{and }\,\,\|f\mid H^{1,p}_z(\Omega)\|:=\||\nabla f|^p \mid H^1_z(\Omega)\|^{\frac{1}{p}}.
\]

Let us prove that a function 
$$
\|\cdot\|_p: f\mapsto  \||\nabla f|^p \mid H^1_r(\Omega)\|^{\frac{1}{p}}
$$
is a seminorm (for the case of $H^1_z(\Omega)$ the proof is similar).

\noindent
1. {\bf Nonnegativity}: 
	$$
	\|f\mid H^{1,p}_r(\Omega)\|:=\||\nabla f|^p \mid H^1_r(\Omega)\|^{\frac{1}{p}}\geq 0\,\,\text{ for all}\,\, f\in H^{1,p}_r(\Omega).
	$$
	
\noindent	
2. {\bf Absolute homogeneity}: 
	\begin{multline*}
	\|kf\mid H^{1,p}_r(\Omega)\|:=\||k \nabla f|^p \mid H^1_r(\Omega)\|^{\frac{1}{p}}=\||k|\cdot |\nabla f|^p \mid H^1_r(\Omega)\|^{\frac{1}{p}}\\
	=
	|k|\||\nabla f|^p \mid H^1_r(\Omega)\|^{\frac{1}{p}}
	=|k|\|f\mid H^{1,p}_r(\Omega)\|
	\end{multline*}
	for any $k\in\mathbb R$ and any $f\in H^{1,p}_r(\Omega)$.

\noindent	
3. {\bf Triangle inequality}:
Let functions $f,g\in H^{1,p}_r(\Omega)$. Then
\begin{align*}
& \|(f+g)\mid H^{1,p}_r(\Omega)\|^{\frac{1}{p}} = \||\nabla f + \nabla g|^p \mid H^1_r(\Omega)\|^{\frac{1}{p}} =\\
& \left(\int\limits_\Omega \sup\limits_{t \leq d(x, \partial \Omega)} \left| \int\limits_{B(x,t)} |\nabla f(y) + \nabla g(y)|^p  \Phi_t(x-t)\, dy \right| \, dx \right)^{\frac{1}{p}} \leq \\
&\left(\int\limits_\Omega \sup\limits_{t \leq d(x, \partial \Omega)} \left| \int\limits_{B(x,t)} (|\nabla f(y)| + |\nabla g(y)|)^p  \Phi_t(x-t)\, dy \right| \, dx \right)^{\frac{1}{p}} = \\
& \left(\int\limits_\Omega \sup\limits_{t \leq d(x, \partial \Omega)} \left| \int\limits_{B(x,t)} \left( \left( \Phi_t(x-t)\right)^{\frac{1}{p}}|\nabla f(y)| 
+ \left( \Phi_t(x-t)\right)^{\frac{1}{p}}|\nabla g(y)| \right)^p\, dy \right| \, dx \right)^{\frac{1}{p}}.
\end{align*}

Now, by using the Minkowski inequality, we have

\begin{multline*}
\left(\int\limits_\Omega \sup\limits_{t \leq d(x, \partial \Omega)} \left| \int\limits_{B(x,t)} \left( \left( \Phi_t(x-t)\right)^{\frac{1}{p}}|\nabla f(y)| 
+ \left( \Phi_t(x-t)\right)^{\frac{1}{p}}|\nabla g(y)| \right)^p\, dy \right| \, dx \right)^{\frac{1}{p}}\\
\leq \left(\int\limits_\Omega \sup\limits_{t \leq d(x, \partial \Omega)} \left( \left| \int\limits_{B(x,t)}  \Phi_t(x-t)|\nabla f(y)|^p \, dy \right|^{\frac{1}{p}} + \left| \int\limits_{B(x,t)}  \Phi_t(x-t)|\nabla g(y)|^p \, dy \right|^{\frac{1}{p}} \right)^p \, dx \right)^{\frac{1}{p}} \\
= \left(\int\limits_\Omega \left( \sup\limits_{t \leq d(x, \partial \Omega)} \left| \int\limits_{B(x,t)}  \Phi_t(x-t)|\nabla f(y)|^p \, dy \right|^{\frac{1}{p}} + \sup\limits_{t \leq d(x, \partial \Omega)} \left| \int\limits_{B(x,t)}  \Phi_t(x-t)|\nabla g(y)|^p \, dy \right|^{\frac{1}{p}} \right)^p \, dx \right)^{\frac{1}{p}} \\
\end{multline*}

Using the Minkowski inequality the second time, we obtain

\begin{flushleft}
$$
\|(f+g)\mid H^{1,p}_r(\Omega)\|^{\frac{1}{p}} = \||\nabla f + \nabla g|^p \mid H^1_r(\Omega)\|^{\frac{1}{p}} \leq
$$
\end{flushleft}
\begin{align*}
\leq \left(\int\limits_\Omega \left( \sup\limits_{t \leq d(x, \partial \Omega)} \left| \int\limits_{B(x,t)}  \Phi_t(x-t)|\nabla f(y)|^p \, dy \right|^{\frac{1}{p}} \right)^p \, dx \right)^{\frac{1}{p}} & \\
 + \left( \int\limits_\Omega \left( \sup\limits_{t \leq d(x, \partial \Omega)} \left| \int\limits_{B(x,t)}  \Phi_t(x-t)|\nabla g(y)|^p \, dy \right|^{\frac{1}{p}} \right)^p \, dx \right)^{\frac{1}{p}} & 
\end{align*}
\begin{align*}
= \left(\int\limits_\Omega \sup\limits_{t \leq d(x, \partial \Omega)} \left| \int\limits_{B(x,t)}  \Phi_t(x-t)|\nabla f(y)|^p \, dy \right| \, dx \right)^{\frac{1}{p}}\\
+ \left( \int\limits_\Omega \sup\limits_{t \leq d(x, \partial \Omega)} \left| \int\limits_{B(x,t)}  \Phi_t(x-t)|\nabla g(y)|^p \, dy \right| \, dx \right)^{\frac{1}{p}} \\
= \|\nabla f \mid H^1_r(\Omega)\|^{\frac{1}{p}} + \|\nabla g \mid H^1_r(\Omega)\|^{\frac{1}{p}}.
\end{align*}

\vskip 0.5cm

\subsection{Duality of Hardy and $\BMO$ spaces}

It is well-known, that dual to the Hardy space $H^1(\mathbb{R}^n)$ is a space $\BMO(\mathbb{R}^n)$, see, for example, \cite{St93}. Recall that a locally integrable function $f: \mathbb{R}^n \to \mathbb{R}$ is a function of bounded mean oscillation ($f \in \BMO(\mathbb{R}^n)$) \cite{FS72} if
$$
\|f \mid \BMO(\mathbb{R}^n)\| := \sup\limits_{B} \frac{1}{B}\int_B |f(x)-f_B|\, dx < \infty,
$$
where the supremum is taken over all balls $B$ in $\mathbb{R}^n$ and $f_B = \frac{1}{B}\int_B f(x)\, dx$.

Since we consider the Hardy spaces defined on Lipschitz domains \cite{CKS93}, we formulate the following version of duality (see \cite{M90, C94, CDS05}. 

Let $\Omega$ be a Lipschitz domain of $\mathbb{R}^n$. 
The space $\BMO_z(\Omega)$ is defined as being the space of all functions in $\BMO(\mathbb{R}^n)$ supported in $\overline{\Omega}$, equipped with the norm
$$
\|f \mid \BMO_z(\Omega)\| := \|f \mid \BMO(\mathbb{R}^n)\|.
$$
The dual of the space $H^1_r(\Omega)$ is the space $\BMO_z(\Omega)$.

The space $\BMO_r(\Omega)$ is defined as being the space of all restrictions to $\Omega$ of functions $\BMO(\mathbb{R}^n)$. It is equipped with the norm
$$
\|f \mid \BMO_r(\Omega)\| := \inf \|F \mid \BMO(\mathbb{R}^n)\|,
$$
where the infimum is taken over all functions $F\in \BMO(\mathbb{R}^n)$ such that $F\vert_{\Omega}=f$. In \cite{J80} it was shown, that $BMO_r(\Omega)$ can be described in another way, namely as a space of locally integrable function on $\Omega$ with
$$
\|f \mid \BMO(\Omega)\| := \sup\limits_{Q} \frac{1}{Q}\int_B |f(x)-f_Q|\, dx < \infty,
$$
where the supremum is taken over all cubes $Q \subset \Omega$ with sides parallel to the axes. Then, the dual of the space $H^1_z(\Omega)$ is $\BMO_r(\Omega)$.

\section{$Q$-quasiconformal mappings}

\subsection{Modulus and capacity}

The theory of $Q$-quasiconformal mappings has been extensively developed in recent decades, see, for example,  \cite{MRSY09}. Let us give the  basic definitions.

The linear integral is denoted by
$$
\int\limits_{\gamma}\rho~ds=\sup\int\limits_{\gamma'}\rho~ds=\sup\int\limits_0^{l(\gamma')}\rho(\gamma'(s))~ds
$$
where the supremum is taken over all closed parts $\gamma'$ of $\gamma$ and $l(\gamma')$ is the length of $\gamma'$. Let $\Gamma$ be a family of curves in $\mathbb R^n$. Denote by $adm(\Gamma)$ the set of Borel functions (admissible functions)
$\rho: \mathbb{R}^n \to[0,\infty]$ such that the inequality
$$
\int\limits_{\gamma}\rho~ds\geqslant 1
$$
holds for locally rectifiable curves $\gamma\in\Gamma$.

Let $\Gamma$ be a family of curves in $\overline{\mathbb R^n}$, where $\overline{\mathbb R^n}$ is a one point compactification of the Euclidean space $\mathbb R^n$. The quantity
$$
M(\Gamma)=\inf\int\limits_{\mathbb{R}^n}\rho^{n}~dx
$$
is called the (conformal) module of the family of curves $\Gamma$ \cite{MRSY09}. The infimum is taken over all admissible functions
$\rho\in adm(\Gamma)$.

Let $\Omega$ be a bounded domain in $\mathbb R^n$ and $F_0, F_1$ be disjoint non-empty compact sets in the
closure of $\Omega$. Let $M(\Gamma(F_0,F_1;\Omega))$ stand for the
module of a family of curves which connect $F_0$ and $F_1$ in $\Omega$. Then \cite{MRSY09}
\begin{equation}\label{eq2}
M(\Gamma(F_0,F_1;\Omega)) = \cp_{n}(F_0,F_1;\Omega)\,,
\end{equation}
where $\cp_{n}(F_0,F_1;\Omega)$ is a conformal capacity of the condensor $(F_0,F_1;\Omega))$ \cite{M}.

Recall that a homeomorphism $\varphi: \Omega\to\widetilde{\Omega}$ of domains
$\Omega,\widetilde{\Omega}\subset \mathbb R^n$ is called a $Q$-homeomorphism \cite{MRSY09}, with a non-negative measurable function $Q$, if
$$
M\left(\varphi \Gamma\right)\leqslant \int\limits_{\Omega} Q(x)\cdot
\rho^{n}(x)dx
$$
for every family $\Gamma$ of rectifiable paths in $\Omega$ and every admissible function $\rho$ for $\Gamma$.

\subsection{Mappings of finite distortion}

Suppose a mapping $\varphi:\Omega\to\mathbb{R}^{n}$ belongs to the class $W^{1,1}_{\loc}(\Omega)$. Then the formal Jacobi
matrix $D\varphi(x)$ and its determinant (Jacobian) $J(x,\varphi)$
are well defined at almost all points $x\in\Omega$. The norm $|D\varphi(x)|$ is the operator norm of $D\varphi(x)$,
i.~e.,  $|D\varphi(x)|=\max\{|D\varphi(x)\cdot h| : h\in\mathbb R^n, |h|=1\}$. We also let $l(D\varphi(x))=\min\{|D\varphi(x)\cdot h| : h\in\mathbb R^n, |h|=1\}$. 

Recall that a Sobolev mapping $\varphi:\Omega\to\mathbb{R}^{n}$ is the mapping of finite distortion if $D\varphi(x)=0$ for almost all $x$ from $Z=\{x\in\Omega: J(x,\varphi)=0\}$ \cite{VGR}. 

Let us define two $p$-distortion functions, $1\leq p<\infty$, for Sobolev mappings of finite distortion $\varphi: \Omega \to \widetilde\Omega$.

\noindent
The outer $p$-dilatation
$$
K_p^O(x,\varphi)=
\begin{cases}
\frac{|D\varphi(x)|^p}{|J(x,\varphi)|},& \,\, J(x,\varphi)\ne 0,\\
0,& \,\, J(x,\varphi)= 0.
\end{cases}
$$
The inner $p$-dilatation
$$
K_p^I(x,\varphi)=
\begin{cases}
\frac{|J(x,\varphi)|}{l(D\varphi(x))^p},& \,\, J(x,\varphi)\ne 0,\\
0,& \,\, J(x,\varphi)= 0.
\end{cases}
$$
Note that $K_n^I(x) \leq (K_n^O(x))^{n-1}$ and $K_n^O(x) \leq (K_n^I(x))^{n-1}$. 

The maximal dilatation, or in short the dilatation, of $\varphi$ at $x$ is defined by
$$
K_p(x) = K_p(x,\varphi) = \max(K_p^O(x,\varphi),K_p^I(x,\varphi)).
$$

Let us recall the weak inverse theorem for Sobolev homeomorphisms \cite{GU10} (see, also \cite{CHM}). 

\begin{thm}
Let $\varphi:\Omega\to\widetilde{\Omega}$, where $\Omega$, $\widetilde{\Omega}$ are domains in $\mathbb R^n$, be a homeomorphism of finite distortion which belongs to the class $W^{1,p}_{\loc}(\Omega)$, $p\geq n-1$, and possesses the Luzin $N$-property (an image of a set of measure zero has measure zero). Then the inverse mapping 
$\varphi^{-1}:\widetilde{\Omega}\to\Omega$ be a mapping of finite distortion which belongs to the class $W^{1,1}_{\loc}(\Omega)$.
\end{thm}

Recall that homeomorphisms $\varphi:\Omega\to\widetilde{\Omega}$ of the class $W^{1,n}_{\loc}(\Omega)$ possess the Luzin $N$-property (an image of a set of measure zero has measure zero) \cite{VGR}.

\section{$\BMO$-quasiconformal mappings and composition operator}

Given a function $Q: \Omega \to [1,\infty]$, a sense-preserving homeomorphism $\varphi: \Omega \to \widetilde{\Omega}$ is called to be $Q$-quasiconformal \cite{MRSY01}, if $\varphi \in W^{1,n}_{\loc}(\Omega)$ and $K_n(x) \leq Q(x)$ for almost all $x\in\Omega$.
If $\varphi$ is $Q$-quasiconformal with $Q\in\BMO_r(\Omega)$, than $\varphi$ is said to be a $\BMO$-quasiconformal mapping.
In \cite{MRSY09}, it was proven that every $\BMO$-quasiconformal mapping is a $Q$-homeomorphism with some $Q \in \BMO_r$.

The first theorem represents a description of composition operators generated by $\BMO$-quasiconformal homeomorphism.

\begin{thm}\label{thm1}
Let $\Omega\subset\mathbb R^n$ be a Lipschitz bounded domain, $\widetilde{\Omega}\subset\mathbb R^n$ be a bounded domain. Suppose there exists $\BMO$-quasiconformal homeomorphism $\varphi: \Omega \to \widetilde{\Omega}$. Then the inverse mapping $\varphi^{-1}: \widetilde{\Omega} \to\Omega$ generates by the composition rule $\left(\varphi^{-1}\right)^{\ast}=f\circ\varphi^{-1}$ a bounded composition operator 
$$
\left(\varphi^{-1}\right)^{\ast}:H^{1,n}_z({\Omega})\cap \Lip ({\Omega}) \to  L^{1,n}(\widetilde{\Omega}),
$$
and the inequality
$$
\|f\circ \varphi^{-1}\mid L^{1,n}(\widetilde{\Omega})\|\leq 
\|Q \mid {\BMO_z({\Omega})}\|^{\frac{1}{n}} \|f  \mid {H^{1,n}_z({\Omega})}\|
$$
holds for any Lipschitz function $f\in \Lip ({\Omega})$.
\end{thm}

\begin{proof} 
Since $\varphi\in  W^{1,n}_{\loc}(\Omega)$ then $\varphi$ possesses the Luzin $N$-property, then the composition $f\circ\varphi^{-1}$ is well defined a.~e. in $\widetilde{\Omega}$. Because $\varphi\in  W^{1,n}_{\loc}(\Omega)$ and has a finite distortion, then $\varphi^{-1}: \widetilde{\Omega} \to\Omega$ belongs to $W^{1,1}_{\loc}(\widetilde{\Omega})$ \cite{GU10}.

Now, let there be given a Lipschitz function $g\in H^{1,n}_z({\Omega})$. Then $g\circ\varphi^{-1}$ is weakly differentiable in $\Omega$, and as long as $\varphi$ has the Luzin $N$-property, the chain rule holds \cite{HK14}. Hence 
$$
\|g \circ \varphi^{-1}\mid L^{1,n}(\widetilde{\Omega})\|^n=\int\limits_{\widetilde{\Omega}}|\nabla g\circ\varphi^{-1}(y)|^n~dy
\leq \int_{\widetilde{\Omega}} |\nabla g|^n(\varphi^{-1}(y))|D\varphi^{-1}(y)|^n\, dy.
$$

By the definition of $\BMO$-quasiconformal mappings there exists measurable function $Q \in \BMO_r({\Omega})$, such that $K_n^I(x)\leq Q(x)$ for almost all $x\in{\Omega}$.
Using the change of variables formula \cite{F69,H93}, we obtain
\begin{multline*}
\int_{\widetilde{\Omega}} |\nabla g|^n(\varphi^{-1}(y))|D\varphi^{-1}(y)|^n\, dy=
\int_\Omega |\nabla g|^n(x)|D\varphi^{-1}(\varphi(x))|^n |J(x,\varphi)|\, dx
\\=
\int_\Omega |\nabla g|^n(x)\frac{|J(x,\varphi)|}{l(D\varphi(x))^n }\, dx\leq \int_\Omega |\nabla g|^n(x)Q(x)~ dx.
\end{multline*}
Now, by the duality of Hardy spaces $H^1_z$ and $\BMO_r$-spaces \cite{C94}, we have
$$
\int_\Omega |\nabla g|^n(x)Q(x)~ dx\leq \|Q \mid {\BMO_r({\Omega})}\| \cdot \|f \mid {H^{1,n}_z({\Omega})}\|^n.
$$
Hence 
$$
\|f \circ \varphi^{-1}\mid L^{1,n}(\widetilde{\Omega})\|\leq \|Q \mid {\BMO_r({\Omega})}\|^{\frac{1}{n}} \|f \mid {H^{1,n}_z({\Omega})}\|
$$
for any Lipschitz function $f\in H^{1,n}_z({\Omega})$.
\end{proof}

Let $\varphi: \Omega \to \widetilde{\Omega}$ be a homeomorphism. Then $\varphi$
is called to be a $\BMO_p$-quasiconformal mapping, if $\varphi \in W^{1,p}_{\loc}(\Omega)$ and $K_p(x) \leq Q(x)$ for almost all $x\in\Omega$ and for some function $Q\in\BMO_r(\Omega)$.

In the case of $\BMO_p$-quasiconformal mappings, we require  additional assumptions on regularity of a mapping $\varphi$ in the case $n-1\leq q<n$.

\begin{thm}\label{thmp}
Let $\Omega\subset\mathbb R^n$ be a Lipschitz bounded domain, $\widetilde{\Omega}\subset\mathbb R^n$ be a bounded domain. Suppose there exists $\BMO_p$-quasiconformal homeomorphism $\varphi: \Omega \to \widetilde{\Omega}$, $p\geq n-1$, which possesses the Luzin $N$--property. Then the inverse mapping $\varphi^{-1}: \widetilde{\Omega} \to\Omega$ generates by the composition rule $\left(\varphi^{-1}\right)^{\ast}=f\circ\varphi^{-1}$ a bounded composition operator 
$$
\left(\varphi^{-1}\right)^{\ast}:H^{1,p}_z({\Omega})\cap \Lip ({\Omega}) \to  L^{1,p}(\widetilde{\Omega}),
$$
and the inequality
$$
\|f\circ \varphi^{-1}\mid L^{1,p}(\widetilde{\Omega})\|\leq 
\|Q \mid {\BMO_r({\Omega})}\|^{\frac{1}{p}} \|f  \mid {H^{1,p}_z({\Omega})}\|
$$
holds for any Lipschitz function $f\in \Lip ({\Omega})$.
\end{thm}

\begin{proof} 
Since $\varphi$ possesses the Luzin $N$-property, then the composition $f\circ\varphi^{-1}$ is well defined a.~e. in $\widetilde{\Omega}$.
Because $\varphi\in  W^{1,p}_{\loc}(\Omega)$, $p\geq n-1$, has a finite distortion and possess the Luzin $N$-property, $\varphi^{-1}: \widetilde{\Omega} \to\Omega$ belongs to $W^{1,1}_{\loc}(\widetilde{\Omega})$ \cite{GU10}.

Now, let there be given a Lipschitz function $g\in H^{1,p}_z({\Omega})$. Then $g\circ\varphi^{-1}$ is weakly differentiable in $\Omega$, and as long as $\varphi$ has the Luzin $N$-property, the chain rule holds \cite{HK14}. Hence 
$$
\|g \circ \varphi^{-1}\mid L^{1,p}(\widetilde{\Omega})\|^p=\int\limits_{\widetilde{\Omega}}|\nabla g\circ\varphi^{-1}(y)|^p~dy
\leq \int_{\widetilde{\Omega}} |\nabla g|^p(\varphi^{-1}(y))|D\varphi^{-1}(y)|^p\, dy.
$$
By the definition of $\BMO$-quasiconformal mappings there exists measurable function $Q \in \BMO_r({\Omega})$, such that $K_p^I(x)\leq Q(x)$ for almost all $x\in{\Omega}$. Using the change of variables formula \cite{F69,H93}, we obtain
\begin{multline*}
\int_{\widetilde{\Omega}} |\nabla g|^p(\varphi^{-1}(y))|D\varphi^{-1}(y)|^p\, dy=
\int_\Omega |\nabla g|^p(x)|D\varphi^{-1}(\varphi(x))|^p |J(x,\varphi)|\, dx
\\=
\int_\Omega |\nabla g|^p(x)\frac{|J(x,\varphi)|}{l(D\varphi(x))^p }\, dx\leq \int_\Omega |\nabla g|^p(x)Q(x)~ dx.
\end{multline*}
Now, by the duality of Hardy spaces $H^1_z$ and $\BMO_r$-spaces \cite{C94}, we have
$$
\int_\Omega |\nabla g|^p(x)Q(x)~ dx\leq \|Q \mid {\BMO_r({\Omega})}\| \cdot \|f \mid {H^{1,p}_z({\Omega})}\|^p.
$$
Hence 
$$
\|f \circ \varphi^{-1}\mid L^{1,p}(\widetilde{\Omega})\|\leq \|Q \mid {\BMO_r({\Omega})}\|^{\frac{1}{p}} \|f \mid {H^{1,p}_z({\Omega})}\|
$$
for any Lipschitz function $f\in H^{1,p}_z({\Omega})$.
\end{proof}

Using the duality between $H^1_r(\Omega)$ and $\BMO_z(\Omega)$ in the same manner we obtain the next two results:

\begin{thm}
Let $\Omega\subset\mathbb R^n$ be a Lipschitz bounded domain, $\widetilde{\Omega}\subset\mathbb R^n$ be a bounded domain. Suppose there exists $\BMO$-quasiconformal homeomorphism $\varphi: \Omega \to \widetilde{\Omega}$ with $Q \in \BMO_z(\Omega)$. Then the inverse mapping $\varphi^{-1}: \widetilde{\Omega} \to\Omega$ generates by the composition rule $\left(\varphi^{-1}\right)^{\ast}=f\circ\varphi^{-1}$ a bounded composition operator 
$$
\left(\varphi^{-1}\right)^{\ast}:H^{1,n}_r({\Omega})\cap \Lip ({\Omega}) \to  L^{1,n}(\widetilde{\Omega}),
$$
and the inequality
$$
\|f\circ \varphi^{-1}\mid L^{1,n}(\widetilde{\Omega})\|\leq 
\|Q \mid {\BMO_z({\Omega})}\|^{\frac{1}{n}} \|f  \mid {H^{1,n}_r({\Omega})}\|
$$
holds for any Lipschitz function $f\in \Lip ({\Omega})$.
\end{thm}

\begin{thm}
Let $\Omega\subset\mathbb R^n$ be a Lipschitz bounded domain, $\widetilde{\Omega}\subset\mathbb R^n$ be a bounded domain. Suppose there exists $\BMO_p$-quasiconformal homeomorphism $\varphi: \Omega \to \widetilde{\Omega}$, $p\geq n-1$, with $Q \in \BMO_z(\Omega)$, which possesses the Luzin $N$--property. Then the inverse mapping $\varphi^{-1}: \widetilde{\Omega} \to\Omega$ generates by the composition rule $\left(\varphi^{-1}\right)^{\ast}=f\circ\varphi^{-1}$ a bounded composition operator 
$$
\left(\varphi^{-1}\right)^{\ast}:H^{1,p}_r({\Omega})\cap \Lip ({\Omega}) \to  L^{1,p}(\widetilde{\Omega}),
$$
and the inequality
$$
\|f\circ \varphi^{-1}\mid L^{1,p}(\widetilde{\Omega})\|\leq 
\|Q \mid {\BMO_z({\Omega})}\|^{\frac{1}{p}} \|f  \mid {H^{1,p}_r({\Omega})}\|
$$
holds for any Lipschitz function $f\in \Lip ({\Omega})$.
\end{thm}

We note the following regularity results also:

\begin{thm}
Given the mapping $\varphi : \Omega \to \widetilde\Omega$. 
\begin{enumerate}
\item If the composition operator $\varphi^\ast : H^{1,p}_r(\widetilde\Omega) \to L^{1,p}(\Omega)$  is bounded, then $\varphi \in L^{1,p}(\Omega)$.
\item If the composition operator $\varphi^\ast : H^{1,p}_r(\widetilde\Omega) \to H^{1,p}_r(\Omega)$ is bounded, then $\varphi \in H^{1,p}_r(\Omega)$.
\end{enumerate}
\end{thm}

\begin{proof}
We prove the theorem only for the first case. The second one is proved in a similar way.

  Due to the boundedness of $\varphi^\ast$
  $$
    \|f\circ\varphi \mid L^{1,p}(\Omega)\| \leq \|\varphi^\ast\| \|f \mid H^{1,p}_z(\widetilde\Omega)\|.
  $$
  Substitute the coordinate functions $f_j = y_j$, $j = 1, ..., n$, we obtain
  \begin{multline*}
    \|f_j \mid H^{1,p}_r(\widetilde\Omega)\| = \int_{\widetilde\Omega} \sup\limits_{0 < t \leq dist(x, \partial \widetilde\Omega)} \left| \frac{1}{t^n} \int_{B(x,t)} \Phi(\frac{x-y}{t}) \right| \, dx \\
	= \int_{\widetilde\Omega} \sup\limits_{0 < t \leq dist(x, \partial \widetilde\Omega)} |1| \, dx = |\widetilde\Omega|.
  \end{multline*}
  Hence,
  $$
    \|f_j\circ\varphi \mid L^{1,p}(\Omega)\| = \|\varphi_j \mid L^{1,p} (\Omega)\| \leq |\widetilde\Omega| \| \varphi^\ast\|.
  $$
\end{proof}

\vskip 0.3cm

{\bf Acknowledgments}:
The first author was supported by the Ministry of Science and higher education of the Russian Federation, state assignment of Sobolev Institute of Mathematics, Project No. 0250-2019-0001.

\vskip 0.3cm

Alexander Menovschikov; Sobolev Institute of Mathematics, 4 Acad. Koptyug avenue, Novosibirsk, 630090, Russia 
 
\emph{E-mail address:} \email{menovschikov@math.nsc.ru} \\

Alexander Ukhlov; Department of Mathematics, Ben-Gurion University of the Negev, P.O.Box 653, Beer Sheva, 8410501, Israel 
							
\emph{E-mail address:} \email{ukhlov@math.bgu.ac.il}

\end{document}